\documentclass[11pt]{amsart}

\usepackage{amsfonts} 
\usepackage{amsmath} 
\usepackage{amssymb} 
\usepackage{amsthm} 
\usepackage{mathtools}
\usepackage{amscd}
\usepackage{booktabs}
\usepackage{color}
\usepackage[dvipsnames]{xcolor}
\usepackage{import}
\usepackage{array}
\usepackage{blkarray}
\usepackage{mathrsfs}

\usepackage[margin=3.5cm]{geometry} 

\usepackage[hidelinks,pagebackref,pdftex]{hyperref} 
\usepackage{makecell} 
\usepackage{multicol}

\usepackage{kbordermatrix} 

\usepackage{marginnote}
\long\def\@savemarbox#1#2{\global\setbox#1\vtop{\hsize\marginparwidth 
  \@parboxrestore\tiny\raggedright #2}}
\renewcommand*{\backref}[1]{}
\renewcommand*{\backrefalt}[4]{
  \ifcase #1
  [No citations.]
  \or [#2]
  \else [#2]
  \fi }

\usepackage{graphicx} 

\AtBeginDocument{%
   \def\MR#1{}
}

\numberwithin{equation}{section}
\theoremstyle{plain}
\newtheorem{theorem}[equation]{Theorem}

\newtheorem{corollary}[equation]{Corollary}
\newtheorem{lemma}[equation]{Lemma}

\newtheorem{proposition}[equation]{Proposition}

\newtheorem*{namedtheorem}{\theoremname}
\newcommand{\theoremname}{testing}

\theoremstyle{definition}

\newcommand{\C}{\mathbb{C}}
\newcommand{\CC}{\mathbb{C}}

\newcommand{\mfc}{\mathfrak{c}}
\newcommand{\mfl}{\mathfrak{l}}
\newcommand{\mfm}{\mathfrak{m}}

\newcommand{\QQ}{\mathbb{Q}}

\newcommand{\TT}{\mathcal{T}}

\newcommand{\Z}{\mathbb{Z}}

\newcommand{\In}{\mathscr{I}}
\newcommand{\NZ}{\mathrm{NZ}}

\newcommand{\PSL}{{\mathrm{PSL}}}
\newcommand{\SL}{{\mathrm{SL}}}

\newcommand{\refsec}[1]{Section~\ref{Sec:#1}}

\newcommand{\reffig}[1]{Figure~\ref{Fig:#1}}

\newcommand{\reflem}[1]{Lemma~\ref{Lem:#1}}
\newcommand{\refprop}[1]{Proposition~\ref{Prop:#1}}
\newcommand{\refthm}[1]{Theorem~\ref{Thm:#1}}
\newcommand{\refcor}[1]{Corollary~\ref{Cor:#1}}

\begin{document}

\title{A-polynomials of fillings of the Whitehead sister} 

\author[J. Howie]{Joshua A. Howie}
\address{School of Mathematics,
Monash University,
VIC 3800, Australia}
\email{josh.howie@monash.edu}

\author[D. Mathews]{Daniel V. Mathews} 
\address{School of Mathematics,
Monash University,
VIC 3800, Australia}
\email{Daniel.Mathews@monash.edu}

\author[J. Purcell]{Jessica S. Purcell}
\address{School of Mathematics,
Monash University,
VIC 3800, Australia}
\email{jessica.purcell@monash.edu}

\author[E. Thompson]{Em K.~Thompson}
\address{School of Mathematics,
Monash University,
VIC 3800, Australia}
\email{em.thompson@monash.edu}


\begin{abstract}
Knots obtained by Dehn filling the Whitehead sister include some of the smallest volume twisted torus knots. Here, using results on A-polynomials of Dehn fillings, we give formulas to compute the A-polynomials of these knots. Our methods also apply to more general Dehn fillings of the Whitehead sister.
\end{abstract}

\maketitle

\section{Introduction}

A-polynomials were introduced in \cite{CCGLS}. They encode information on the deformation space of hyperbolic structures of knots, on incompressible surfaces embedded in the knot complements, on volumes and cusp shapes. They also play into conjectures in quantum topology, such as the AJ-conjecture~\cite{FrohmanGelcaLofaro, garoufalidis:AJ-conj, garoufalidisLe, GK74}.

In general, it is a difficult problem to compute explicit formulas for A-polynomials of families of knots. However, explicit or recursive formulas are known for some simple families of knots. Recursive formulas for A-polynomials were first given for twist knots, by Hoste and Shanahan~\cite{HosteShanahan04}, and in closed form by Mathews~\cite{Mathews:A-poly_twist_knots_err, Mathews:A-poly_twist_knots}. Formulas for $(-2,3,2n+1)$-pretzel knots were found by Tamura and Yokota~\cite{TamuraYokota}, and by Garoufalidis and Mattman~\cite{GaroufalidisMattman}.
Petersen found A-polynomials of certain double twist knots $J(k,\ell)$~\cite{Petersen:DoubleTwist}, recovering and extending Hoste and Shanahan's work. Closed form formulas for knots with Conway's notation $C(2n,3)$ were given by Ham and Lee~\cite{HamLee}, and Tran found formulas for A-polynomial 2-tuples of a family of 2-bridge links he calls twisted Whitehead links~\cite{Tran}. A-polynomials of cabled knots and iterated torus knots were given by Ni and Zhang~\cite{NiZhang}.

In \cite{HowieMathewsPurcell}, it was shown that the A-polynomial could be defined by quadratic polynomials obtained from a triangulation of the knot complement, with particularly simple form for families of knots obtained by Dehn filling a parent link. Of the known examples above, all the hyperbolic families are obtained by simple Dehn fillings of a parent link. In each of these cases, the $n$th knot in the family differs from the $(n-1)$th by adding exactly two crossings to a twist region; in particular the method of \cite{HowieMathewsPurcell} applies. However, the methods of \cite{HowieMathewsPurcell} also apply more broadly. In this paper, we apply them to a family of twisted torus links. This family is unlike those above in that changing the Dehn filling slope adjusts the diagram by adding twenty crossings rather than just two in a twist region. Thus techniques to compute A-polynomials using diagrams, or group presentations coming from diagrams, would be difficult to apply to this family of knots. 


The family that we consider are knots obtained by Dehn filling the Whitehead sister. 
The Whitehead sister is known to be the complement of the $(-2,3,8)$-pretzel link, shown on the left of \reffig{WhSisLink}. An equivalent link is shown on the right. It has two components; one component is an unknot in $S^3$. Hence when we perform $1/n$-Dehn filling of the unknotted link component, we obtain a knot complement in $S^3$. In fact, as indicated by the form of the link on the right of \reffig{WhSisLink}, the $1/n$-Dehn filling is the twisted torus knot $T(5,1-5n,2,2)$, with notation as in~\cite{ckm}. When $n=1$, this is the $(-2,3,7)$-pretzel knot.

\begin{figure}
  \includegraphics{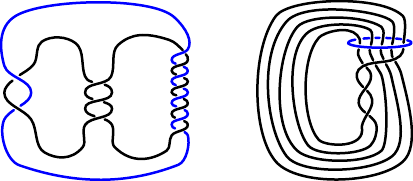}
  \caption{Two views of the $(-2,3,8)$-pretzel link. One component (shown in blue) is an unknot embedded in $S^3$.}
  \label{Fig:WhSisLink}
\end{figure}

The $(-2,3,8)$-pretzel link complement is constructed by face pairings of a single regular ideal octahedron. It is known to be one of the two hyperbolic 3-manifolds of minimal volume with exactly two cusps, by work of Agol~\cite{Agol:MinVolume}. The other minimal volume 2-cusped hyperbolic 3-manifold is the Whitehead link complement, which is also constructed by face-pairings of a regular ideal octahedron. For this reason, the complement of the $(-2,3,8)$-pretzel link is known as the Whitehead sister.
In SnapPy~\cite{SnapPy}, additional names for this 3-manifold are {\texttt{m125}}, {\texttt{ooct01\_0000}}, and the link complement with the same framing is {\texttt{L13n5885}}. Alternatively, it may be obtained by $-3/2$-Dehn filling on any one component of the 3-component link known as the \emph{magic manifold}. Its exceptional Dehn fillings have been completely classified by Martelli and Petronio~\cite{MartelliPetronio:Magic}; there are exactly six of them.

In this paper, we obtain triangulations of all but three, and A-polynomials of all but seven Dehn fillings of the Whitehead sister. Triangulations use the layered solid tori of Jaco and Rubinstein~\cite{JacoRubinstein:LST}; A-polynomial calculations apply the methods in \cite{HowieMathewsPurcell}. 
Many of these manifolds are recognised to be manifolds in the census of cusped 3-manifolds obtained by small numbers of tetrahedra~\cite{Census}, including knot complements~\cite{CallahanDeanWeeks, ckm, ckp}.

Our main result concerns the A-polynomials of the $1/n$-Dehn fillings. 
Throughout, when we state that we are performing a Dehn filling of the Whitehead sister, we mean that we are performing the filling along the unknotted component of the $(-2,3,8)$-pretzel link. We use the terminology \emph{Whitehead sister} to refer to the 3-manifold that is the complement of this link.

\begin{theorem}\label{Thm:Main}
For $n\geq 3$, suppose $K(n)$ is the knot obtained by $1/n$-Dehn filling of the Whitehead sister. 
Then (a factor of) the $\PSL(2,\CC)$ A-polynomial of $K(n)$ is obtained from the following set of equations after eliminating all variables except $\ell$ and $m$.

Outside equations: \hspace{.2in}
$\displaystyle{
  \gamma_{4/1} = \frac{-\ell+m}{\sqrt{\ell}(-1+m)}, \quad
  \gamma_{1/0} = \frac{-\ell+m^2}{\sqrt{m}(-\ell+m)} 
}$

First inside equations:
\begin{gather*}
  \gamma_{2/1} = (\gamma_{1/0}^2-1)/\gamma_{4/1}, \quad
  \gamma_{1/1} = \gamma_{2/1}^2 - \gamma_{1/0}^2, \quad
  \gamma_{0/1} = (\gamma_{1/1}^2-\gamma_{1/0}^2)/\gamma_{2/1} \\
  \gamma_{1/2}=(\gamma_{0/1}^2 - \gamma_{1/1}^2)/\gamma_{1/0}, 
\end{gather*}

Recursive inside equations (empty if $n=3$):
\[
  \gamma_{1/(k-1)}\gamma_{1/(k-3)} + \gamma_{0/1}^2 - \gamma_{1/(k-2)}^2=0, \mbox{ for } 4\leq k \leq n.
\]

And the folding equation: $\gamma_{0/1}=\gamma_{1/(n-1)}$.
\end{theorem}

\refthm{Main} effectively gives A-polynomials explicitly: eliminating all $\gamma$ variables involves only substitution. 
The outside equations express $\gamma_{1/0}$ and $\gamma_{4/1}$ in terms of $\ell$ and $m$; each inside equation expresses a $\gamma$ variable 
in terms of previous $\gamma$ variables, hence in terms of $\ell$ and $m$; the folding equation equates two expressions in $\ell$ and $m$, which suitably rearranged gives the A-polynomial. 
Note implicit in the statement of the theorem, and in the process of elimination of $\gamma$ variables, is that the variables $\gamma_i$ are nonzero; this holds because they are exponentials of other variables in~\cite{HowieMathewsPurcell}, so we assume throughout that $\gamma$ variables are never zero.

After elimination, we obtain the same factor of the A-polynomial as Champanerkar~\cite{Champanerkar:Thesis}. In particular, the factor corresponding to the complete hyperbolic structure is a factor of the polynomial of \refthm{Main}. However, the A-polynomial may have additional factors that the gluing variety does not pick up; see Segerman~\cite{Segerman:Deformation}.

For convenience, we have only stated the result for $1/n$-fillings for $n \geq 3$ here. A corresponding result for negative $n$ (specifically, $n \leq -2$) is \refthm{RecursiveLSTNegative}. Equations for the remaining $n \neq 0$ are easily found using the methods described in this paper.

Indeed, \refcor{DehnFillWhSis} gives the A-polynomial of a general Dehn filling similarly explicitly, for \emph{any} slope $p/q$ except those in $\{2, 3, 7/2, 11/3, 4, 5, 1/0 \}$. 

The ``missed" Dehn filling slopes arise for two reasons. First, our methods are not guaranteed to apply to non-hyperbolic Dehn fillings, which have slopes $p/q\in\{2,3,7/2, 11/3, 4, 1/0\}$ using our framing (which is different from that of Martelli and Petronio \cite{MartelliPetronio:Magic}, but yields the same exceptions). Second, although the methods of \cite{HowieMathewsPurcell} can deal with all hyperbolic fillings, they involve degenerate layered solid tori for slopes $p/q \in \{ 2/1, 7/2, 5/1\}$; we omit them here.

A triangulation of a hyperbolic 3-manifold $M$ is \emph{geometric} if the hyperbolic structure on $M$ is built by putting a positively oriented hyperbolic structure on each tetrahedron and then gluing. A triangulation is \emph{minimal} if $M$ cannot be triangulated by fewer tetrahedra. 
For $n\in \{\pm 1, \pm 2, \pm 3, \pm 4\}$ the $1/n$-Dehn filling of the Whitehead sister appears in the SnapPy census. Thus we know these triangulations are both geometric and minimal. The Whitehead sister also satisfies conditions required by Gu{\'e}ritaud and Schleimer to ensure that its sufficiently high Dehn fillings are geometric~\cite{GueritaudSchleimer}; this means there exists some $N$ such that for $n\geq N$, the $1/n$-Dehn filling described here is geometric. Unfortunately the bound on $N$ from \cite{GueritaudSchleimer} is not explicit. 
We conjecture that the triangulations of all the $1/n$-Dehn fillings in this paper are both geometric and minimal. 

We note that Thompson has been able to use the formulas of~\cite{HowieMathewsPurcell}, along with results in cluster algebras, to give more explicit closed forms for the A-polynomials of the knots in this paper~\cite{Thompson:Cluster}.

\subsection{Acknowledgements}
This work was partially funded by the Australian Research Council, grant DP210103136. Thompson was supported by an Australian Government Research Training Program (RTP) scholarship. 

\section{Background on A-polynomials}
\label{Sec:APoly}

Suppose a compact 3-manifold has boundary consisting of a single torus, and its interior admits a complete hyperbolic structure. Thurston observed that such a manifold has a 2-(real) dimensional space of incomplete hyperbolic structures~\cite{Thurston:Notes}. In the concrete setting of the figure-8 knot complement, Thurston showed that the complete hyperbolic structure is obtained by triangulating the knot complement by two regular ideal tetrahedra, and that incomplete structures are obtained by deforming the hyperbolic structures on the ideal tetrahedra in a neighbourhood of the complete structure. For a general hyperbolic 3-manifold with a decomposition into hyperbolic ideal tetrahedra, the space of deformations of the hyperbolic structures on ideal tetrahedra is now known as the \emph{deformation variety} or the \emph{gluing variety}, because the tetrahedra are required to satisfy gluing equations. The gluing variety encodes incomplete hyperbolic structures, and also additional information whose geometric interpretation is not clear. The face pairings of the tetrahedra at a point in the gluing variety will give a representation of the fundamental group of the 3-manifold into $\PSL(2,\CC)$. 

Culler and Shalen considered representations of the fundamental group of a 3-manifold into $\SL(2,\CC)$, and put them into an algebro-geometric framework. Such representations form an $\SL(2,\CC)$ character variety~\cite{CullerShalen:Varieties}. In the case of a hyperbolisable 3-manifold with a single cusp (i.e.\ the interior of a compact 3-manifold with a single torus boundary component), the $\SL(2,\CC)$ character variety will be 2-(real) dimensional, which is implied by Thurston's work. In~\cite{CCGLS}, Cooper, Culler, Gillet, Long, and Shalen introduce the A-polynomial. This polynomial gives a description of the 2-dimensional representation variety in terms of the variables $M$ and $L$, which in their setting are eigenvalues of matrices representing meridian and longitude curves in the fundamental group of the torus boundary component of the original compact 3-manifold. An A-polynomial can also be defined when $\PSL(2,\CC)$ representations are used, and this was further investigated by Boyer and Zhang~\cite{BoyerZhang}.  

Returning to triangulations, certain products of parameters that encode meridian and longitude, known as the \emph{cusp equations}, will be trivial in the complete setting. 
Champanerkar observed that by writing the cusp equations in variables $m$ and $\ell$, one could obtain a polynomial describing the gluing variety~\cite{Champanerkar:Thesis}. Champanerkar proved that the polynomial obtained by this method will always divide the $\PSL(2,\CC)$ A-polynomial. However, the gluing and cusp equations required for Champanerkar's method are often very high degree in a number of variables $z_i$ corresponding to the number of ideal tetrahedra, and obtaining this A-polynomial requires simultaneously eliminating variables $z_i$ to reduce to a single polynomial in $m$ and $\ell$. This is not always possible even with computer assistance. 

In~\cite{HowieMathewsPurcell}, Howie, Mathews and Purcell use ideas of Dimofte~\cite{DimofteQRCS} and results of Neumann and Zagier~\cite{NeumannZagier} to change the variables of Champanerkar's equations. This produces Ptolemy-like equations. 
Instead of using variables associated to ideal tetrahedra, the variables are associated to edges of the triangulations, with one equation per tetrahedron. 
The new variables $\gamma_i$ are exponentials of variables $\Gamma_i$ arising from the linear algebra of a symplectic extension of the Neumann--Zagier matrix, so are never zero.
In~\cite{HowieMathewsPurcell}, it was shown that these equations can lead to a simpler system of equations for 3-manifolds obtained by Dehn filling; this was further explored by Thompson~\cite{Thompson:Cluster}. These are the results that we make use of in this paper. 

\subsection{Notes on the variables $m$ and $\ell$}
In this paper, $m$ and $\ell$ have a geometric meaning: they come from identifications of tetrahedra whose ideal vertices form a meridian or longitude of the cusp boundary. This is the same meaning as in work of Champanerkar~\cite{Champanerkar:Thesis}. 
However, note that in \refthm{Main} there are square roots of $m$ and $\ell$ involved in the defining equations, where Champanerkar only includes integer powers of $m$ and $\ell$. 
This is a consequence of the construction of~\cite{HowieMathewsPurcell}. To move from the traditional description of the gluing variety to the Ptolemy-like description, we change variables by inverting a matrix that is an expanded symplectic version of the Neumann--Zagier matrix. Neumann and Zagier showed that equations in tetrahedra parameters defining $m$ and $\ell$, as well as gluing equations, have a symplectic-like structure~\cite{NeumannZagier} given by a symplectic pairing $\omega(\cdot, \cdot)$. Under the pairing, vectors obtained from gluing equations give zero. Vectors obtained from curves on the cusp give \emph{twice} the intersection number. It is this factor of two --- twice the intersection number --- that introduces the square roots into our equations. 
The square roots can be cleared by rationalisation, but the resulting equations are more complicated, and so we leave them as they are. 

The variables $M$ and $L$ in the traditional A-polynomial correspond to eigenvalues of matrices, and do not have the same geometric meaning as $m$ and $\ell$ here. However, they are related by $M^2 = m$ and $L^2 = \ell$; see~\cite[Corollary~1.4]{HowieMathewsPurcell}. 

\subsection{Comparison with other Ptolemy equations}
Garoufalidis, D.~Thurston, and Zickert define a Ptolemy variety~\cite{GTZ}, inspired by work of Fock and Goncharov~\cite{FockGoncharov06}.
This assigns a Ptolemy relation to tetrahedra in a triangulation of a 3-manifold and leads to a representation of the fundamental group into $\SL(2,\CC)$. In this setting, a divisor of the $\SL(2,\CC)$ A-polynomial is also obtained; see Zickert~\cite[Corollary~1.7]{Zickert:PtolemyDehnA-poly}, and Goerner and Zickert~\cite{GoernerZickert}. 
We conjecture that there is a geometric connection between the methods of~\cite{GTZ} and the methods here. However, this is not clear \emph{a priori}. In~\cite{GTZ}, Ptolemy equations are obtained combinatorially from oriented tetrahedra. In this paper, orientation is not required.

\section{Triangulation of the Whitehead sister}
\label{Sec:triangulation_of_Whitehead_sister}

The default SnapPy triangulation of the Whitehead sister {\texttt{m125}} has four tetrahedra, with an ideal edge added to subdivide the ideal octahedron. 
There are three choices for adding such an ideal edge, and in the SnapPy census, it is chosen so that it meets both cusps. We wish instead to choose an edge that does not meet the cusp corresponding to the unknotted component, as this will make it simpler to triangulate Dehn fillings. 
\reffig{WhiteheadSisterTable} gives our triangulation. 
The notation, as in Regina, is as follows. The four tetrahedra are labeled $0, 1, 2, 3$; each has ideal vertices labeled $0, 1, 2, 3$. The top-left entry $2(312)$ says that the face of tetrahedron $0$ with vertices $012$ is glued to the face of tetrahedron $2$ with vertices $312$, with the ideal vertices glued in order. 
Up to relabeling, this is obtained from the default SnapPy triangulation of {\texttt{m125}} by performing a 4-4 move.

\begin{figure}
  \begin{tabular}{c|c|c|c|c}
    Tetrahedron & Face 012 & Face 013 & Face 023 & Face 123 \\
    \hline
    0 & 2(312) & 1(023) & 1(312) & 1(031) \\
    1 & 3(123) & 0(132) & 0(013) & 0(230) \\
    2 & 3(021) & 3(031) & 3(032) & 0(120) \\
    3 & 2(021) & 2(031) & 2(032) & 1(012) \\
  \end{tabular}
  \caption{Four-tetrahedron triangulation of the Whitehead sister.}
  \label{Fig:WhiteheadSisterTable}
\end{figure}

An embedded horospherical torus about a cusp intersects the ideal tetrahedra in triangles, inducing a \emph{cusp triangulation}.
The cusp to be filled meets only tetrahedra 2 and 3, each in the vertex labeled 0, giving a cusp triangulation with two triangles. The cusp triangulation for other cusp is shown in \reffig{WhSisCuspTri}. Tetrahedra $2$ and $3$ form a hexagon within this cusp triangulation. We choose generators $\mathfrak{l}, \mathfrak{m}$ of the homology of this cusp that avoid the hexagon and meet as few cusp triangles as possible.

\begin{figure}
  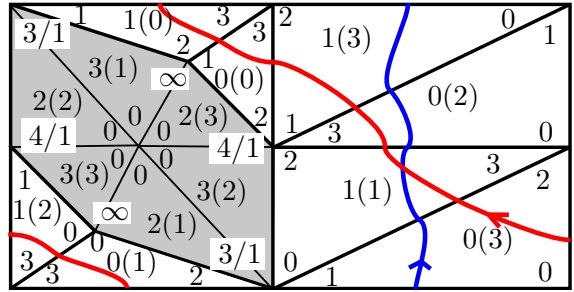
  \caption{Triangulation of the unfilled cusp of the Whitehead sister.  
Generators of homology are shown: $\mathfrak{l}$ in red and $\mathfrak{m}$ in blue.}
  \label{Fig:WhSisCuspTri}
\end{figure}

The triangulation has exactly four edge classes. One edge, which we will call $e$, runs from one cusp to the other. It has one end in the centre of the shaded hexagon. The other three edges have both of their ideal endpoints on the unfilled cusp. All three lie on the boundary of tetrahedra $2$ and $3$. They are labeled $3/1$, $4/1$, and $\infty$ in the figure, for reasons we will explain below.

\subsection{Meridian and longitude basis}
We also need to identify the actual meridian and preferred longitude for the cusp. We can use SnapPy to determine these, either using the PLink editor to enter the pretzel link $P(-2,3,8)$ into SnapPy or using {\texttt{L13n5885}}; this input ensures treatment as a link complement in $S^3$.
We find that one of the generators we chose, namely $\mathfrak{m}$, was indeed the meridian. The preferred longitude $\mathfrak{l}'$ is shown in Figure~\ref{Fig:preflong}. We have $\mathfrak{l}'=\mathfrak{l}\mathfrak{m}^{-8}$.

\begin{figure}
  \centering
  \includegraphics{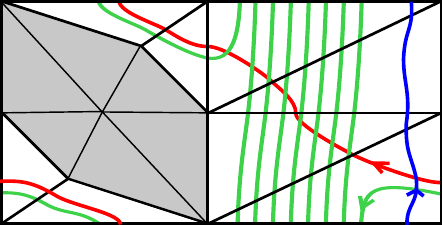}
  \caption{The preferred longitude $\mathfrak{l}' =\mathfrak{l}\mathfrak{m}^{-8}$ is shown in green.  }
    \label{Fig:preflong}
\end{figure}

\section{Dehn filling triangulations}

To perform Dehn filling, pull out tetrahedra $\Delta_2$ and $\Delta_3$. The union of these two tetrahedra is homeomorphic to $T^2\times[0,\infty)$ with a single point removed from its boundary $T^2\times\{0\}$. Its complement is built by gluing tetrahedra $\Delta_0$ and $\Delta_1$, but with two faces unglued, namely face $012$ of tetrahedron $\Delta_0$ and face $012$ of tetrahedron $\Delta_1$. Its boundary is a punctured torus triangulated by these two faces.

The Dehn filling is obtained by attaching a triangulated solid torus to these two faces. That is, we build a solid torus whose boundary is triangulated by the ideal triangles corresponding to the unglued faces of $\Delta_0$ and $\Delta_1$. The meridian of the solid torus gives the slope of the Dehn filling. 

To describe the slope of the Dehn filling, let $\mu$, $\lambda$ denote the standard meridian, longitude pair for an unknotted component. Dehn filling along any slope of the form $\mu + n\lambda$ for $n\in \Z$ will result in the complement of a knot in $S^3$. More generally, write any slope $p\mu + q\lambda$ for $p,q\in \Z$ by $p/q\in\QQ\cup\{1/0\}$. 

\subsection{Layered solid tori}
We use the \emph{layered solid torus} construction of ~\cite{GueritaudSchleimer, JacoRubinstein:LST} to fill a specified slope $r=p/q$. The boundary of a layered solid torus is a 1-punctured torus, triangulated by two fixed ideal triangles. Edges of the boundary triangles form slopes on the 1-punctured torus, each of which can be written as some $a/b\in\QQ\cup\{1/0\}$. A triangulation of a 1-punctured torus consists of a triple of slopes for which the geometric intersection number of any pair is $1$. These are encoded by the \emph{Farey triangulation}.

\begin{figure}
  \includegraphics{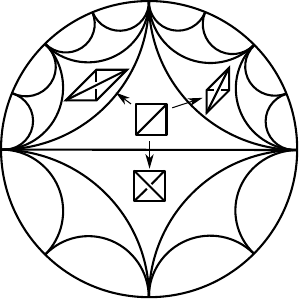}
  \caption{Tetrahedra obtained from moving across triangles in the Farey triangulation.}
  \label{Fig:FareyLST}
\end{figure}

In our case, initial triangles corresponding to faces of $\Delta_0$ and $\Delta_1$ give a starting triangle in the Farey triangulation. There is a geodesic from this triangle to the rational number $r=p/q$, which alternatively can be considered as a length-minimising path through the dual 1-skeleton of the Farey triangulation. The sequence of triangles meeting the geodesic gives a sequence of triangulations of a 1-punctured torus, each obtained from the previous by a diagonal exchange. The diagonal exchange can be realised by layering a tetrahedron onto the punctured torus; see \reffig{FareyLST}, taken from \cite{HowieMathewsPurcell}.

Layering tetrahedra in this manner builds a space homotopy equivalent to a thickened punctured torus.
At each step, the space has two boundary components. One is marked with the initial triangulation. The other is marked with the triangulation in the Farey graph corresponding to the most recently added tetrahedron.

To obtain a solid torus with the appropriate meridian, we stop layering tetrahedra after reaching the triangle previous to the one containing $r$.
Being separated from $r$ by a single edge of a Farey triangle, a diagonal exchange at this point would give a triangulation with slope $r$;
but instead of a diagonal exchange, we fold the two triangles across the corresponding diagonal. This gives a manifold homotopy equivalent to a solid torus, and makes the slope $r$ homotopically trivial. See \reffig{FoldTriangles}, taken from \cite{HowieMathewsPurcell}.

\begin{figure}
  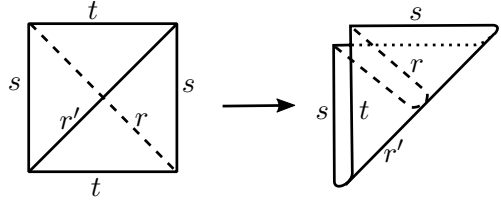
  \caption{Folding makes the diagonal slope $r$ homotopically trivial.}
  \label{Fig:FoldTriangles}
\end{figure}

Applying this procedure to the Whitehead sister, we first find the initial slopes. 
After removing tetrahedra $\Delta_2$ and $\Delta_3$, the resulting punctured torus boundary is triangulated by three slopes, which (using SnapPy~\cite{SnapPy} and Regina~\cite{Regina}) we find to be $4/1$, $3/1$, and $1/0$. Different manifolds obtained by Dehn filling are shown in the Farey graph in \reffig{FareyWhSis}. Observe that aside from the first step, each step in the Farey triangulation can be labeled with an L, for turning left, or R, for turning right. 

\begin{figure}
  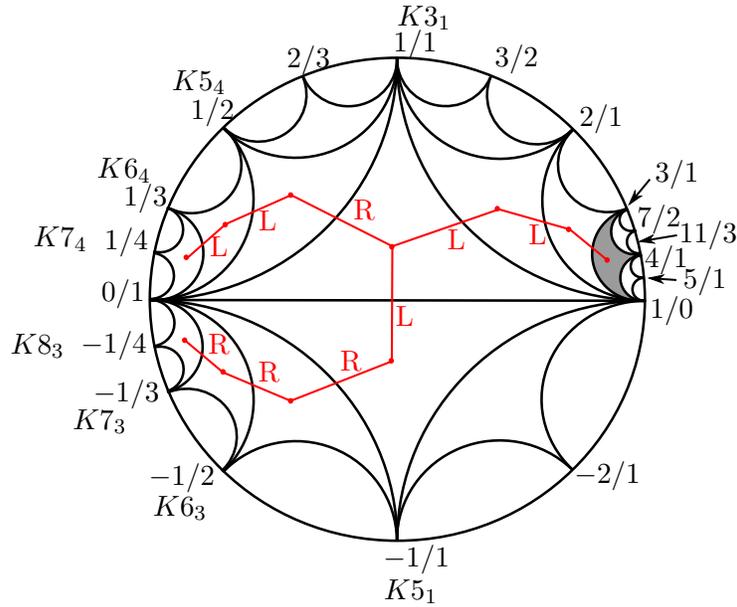
  \caption{Paths in the Farey triangulation that produce knots obtained by Dehn filling the Whitehead sister. 
	}
  \label{Fig:FareyWhSis}
\end{figure}

We may now compute triangulations of Dehn fillings. To obtain two infinite families of knots, perform $1/n$-Dehn fillings for positive and negative integers $n$. When $n$ is positive, these include hyperbolic knot complements $K3_1$, $K5_4$, $K6_4$, $K7_4$, and $K8_4$.
When $n$ is negative, these include hyperbolic knot complements $K5_1$, $K6_3$, $K7_3$, and $K8_3$. These fillings correspond to paths in the Farey triangulation that start at $3/1, 4/1, 1/0$, move towards $2/1$, then step L, L; in the positive case this is followed by an R and a sequence of L's; in the negative case this is followed by an L and a sequence of R's. See \reffig{FareyWhSis}.

For example, $K3_1$ is obtained by removing $\Delta_2$ and $\Delta_3$ from the Whitehead sister manifold, and attaching a single tetrahedron and then folding. 
Figure~\ref{Fig:WhSis3TetCusps} shows cusp triangulations of Dehn fillings producing $K5_4$ and $K5_1$.

\begin{figure}
  \includegraphics{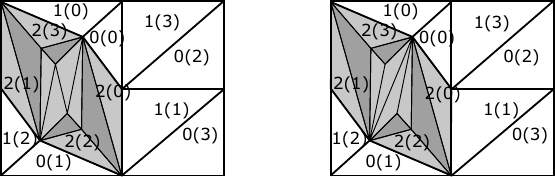}
  \caption{Cusp triangulations of manifolds obtained by $1/2$ and $-1/1$ Dehn fillings. The shaded triangles come from the layered solid torus.}
  \label{Fig:WhSis3TetCusps}
\end{figure}

\subsection{Exceptional manifolds}

As mentioned in the introduction, the methods of \cite{HowieMathewsPurcell} only apply to hyperbolic fillings, which exclude the initial slopes $3/1$, $4/1$, $1/0$, along with $2/1$, $7/2$ and $11/3$. 
We also ignore what we call
the degenerate Dehn fillings.
These are the Dehn fillings for which we do not add any new tetrahedra to perform Dehn filling, but merely remove the two tetrahedra corresponding to the cusp and then fold as in \reffig{FoldTriangles}. There are three such Dehn fillings, corresponding to the three slopes in the Farey triangulation that lie in triangles sharing an edge with our initial triangle in \reffig{FareyWhSis}. These are the slopes $2/1$, $7/2$, and $5/1$; see \reffig{FareyWhSis}. 
Of these, $2/1$ and $7/2$ are not hyperbolic, so not relevant. 
The slope $5/1$ is the manifold {\texttt{m003}}, or the figure-8 sister, built of two regular ideal tetrahedra. Its $\PSL(2,\CC)$ A-polynomial can be computed by hand, or by noting that {\texttt{m003}} is also homeomorphic to the Dehn filling along slope $10/3$.

\section{A-polynomial equations}

Suppose a knot complement is triangulated by $n$ ideal tetrahedra. Label the ideal vertices of each tetrahedron $0,1,2,3$ so that, when viewed from ideal vertex $0$, ideal vertices $1,2,3$ appear in anticlockwise order. We refer to the edges between $0,1$ and $2,3$ as $a$-edges, between $0,2$ and $1,3$ as $b$-edges, and between $0,3$ and $1,2$ as $c$-edges. The $a,b,c$-edges of the $i$th tetrahedron are called $a_i,b_i,c_i$-edges.

We use the deformation variety as in~\cite{Champanerkar:Thesis} to compute the $\PSL(2,\CC)$ A-polynomial. This variety is cut out by \emph{gluing} and \emph{completeness} equations, which can be read off of the Neumann--Zagier matrix. 

In order to determine the Neumann--Zagier matrix for a triangulation, we first define the \textit{incidence matrix} $\In$.
This matrix has a row for each edge class of the triangulation, and a row for each generator of cusp homology; it has three columns for each tetrahedron of the triangulation, labelled $a_i,b_i,c_i$. 
Thus for the Whitehead sister, $\In$ has 4 edge rows, 4 cusp rows, and 4 triples of columns.
Entries in edge rows count the number of $a_i$-, $b_i$- and $c_i$-edges incident with that edge. 
For each cusp we choose oriented representatives $\mathfrak{m}, \mathfrak{l}$ of generators that intersect edges in the cusp triangulation transversely, and so that the algebraic intersection number of $\mathfrak{m}$ and $\mathfrak{l}$ is $1$.
The entries in the cusp rows count the number of $a_i$-, $b_i$- and $c_i$-edges cut off by $\mathfrak{m}$ and $\mathfrak{l}$, with edges to the left counted with $+1$ and edges to the right counted with $-1$.
In our case, the two have curves 
$\mathfrak{m}_0$, $\mathfrak{l}_0$ shown in \reffig{WhSisCuspTri} 
and $\mathfrak{m}_1, \mathfrak{l}_1$ shown in \reffig{MeridLongWhSis}.

\begin{figure}
  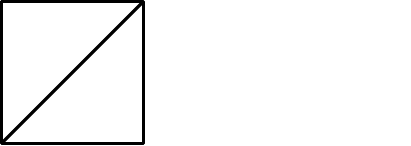
  \caption{Choices for $\mfm_1$ and $\mfl_1$.}
  \label{Fig:MeridLongWhSis}
\end{figure}

The Neumann--Zagier matrix $\NZ$ is obtained from $\In$ by replacing the $a_i, b_i, c_i$ columns with two columns, subtracting the $c_i$ from the $a_i$ and $b_i$ columns so as to obtain $a_i - c_i$ and $b_i - c_i$ entries. 
Additionally, form a vector $C$ with the same number of rows as $\In$, which is obtained by subtracting all the coordinates in $c$-columns from the vector consisting of $2$s for edge rows and $0$s for cusp rows.  

For the Whitehead sister, we obtain $\In, \NZ$ and $C$ as in \reffig{MatricesWhiteadSister}.

\begin{figure}
\[
\In = \kbordermatrix{
    &  & \Delta_0 &  & \vrule & & \Delta_1 & & \vrule  & & \Delta_2 & & \vrule  & & \Delta_3 & \\
  E_{3/1}  & 0 & 1 & 2 & \vrule  & 2 & 0 & 1 & \vrule & 0 & 0 & 1 & \vrule & 0 & 0 & 1 \\
  E_{4/1}  & 1 & 1 & 0  & \vrule & 0 & 1 & 1  & \vrule & 1 & 0 & 0  & \vrule & 1 & 0 & 0  \\
  E_{1/0}  & 1 & 0 & 0  & \vrule & 0 & 1 & 0  & \vrule & 0 & 1 & 0  & \vrule & 0 & 1 & 0 \\
  E_e     & 0 & 0 & 0  & \vrule & 0 & 0 & 0  & \vrule & 1 & 1 & 1  & \vrule & 1 & 1 & 1 \\
  \hline
  \mfm_0  & 1 & 1 & 0  & \vrule & 0 & -1 & -1  & \vrule & 0 & 0 & 0  & \vrule & 0 & 0 & 0 \\
  \mfl_0  & 0 & 1 & -1  & \vrule & 1 & 0 & -1  & \vrule & 0 & 0 & 0  & \vrule & 0 & 0 & 0 \\
  \mfm_1 & 0 & 0 & 0  & \vrule  & 0 & 0 & 0  & \vrule & 1 & 0 & 0 & \vrule & -1 & 0 & 0 \\
  \mfl_1 & 0 & 0 & 0  & \vrule & 0 & 0 & 0  & \vrule & 0 & 1 & 0 & \vrule & 0 & -1 & 0
}
\]

\[
\NZ = \kbordermatrix{
          &   &\Delta_0 & \vrule & &\Delta_1 & \vrule &  & \Delta_2 & \vrule & &\Delta_3 \\
  E_{3/1}  &-2& -1 &\vrule &1 &-1 & \vrule  & -1 & -1 &\vrule  & -1& -1 \\
  E_{4/1}  & 1 & 1 &\vrule & -1 & 0 &\vrule & 1 & 0 & \vrule & 1 & 0 \\
  E_{1/0}  & 1 & 0 & \vrule & 0 & 1 &\vrule & 0 & 1 & \vrule & 0 & 1  \\
  E_e     & 0 & 0 &\vrule & 0 & 0 &\vrule & 0 & 0 &\vrule & 0 & 0 \\
  \hline
  \mfm_0  & 1 & 1 &\vrule & 1 & 0 &\vrule & 0 & 0 &\vrule & 0 & 0 \\
  \mfl_0  & 1 & 2 &\vrule & 2 & 1 &\vrule & 0 & 0 &\vrule & 0 & 0 \\
  \mfm_1  & 0 & 0 &\vrule & 0 & 0 &\vrule & 1 & 0 &\vrule & -1 & 0 \\
  \mfl_1  & 0 & 0 &\vrule & 0 & 0 &\vrule & 0 & 1 &\vrule & 0 & -1
}
\quad
C = 
\begin{bmatrix}
  -3 \\
  1 \\
  2 \\
  0 \\
  \hline
  1 \\
  2 \\
  0 \\
  0
\end{bmatrix}
\]

  \caption{Incidence and Neumann-Zagier matrices for the Whitehead sister.}
  \label{Fig:MatricesWhiteadSister}
\end{figure}

By \cite[Lemma~3.5]{HowieMathewsPurcell}, using work of Neumann~\cite{Neumann}, there exists an integer vector $B$ such that $\NZ\cdot B = C$, and such that the last entries of $B$, corresponding to the two tetrahedra meeting the second cusp, are all zeros. For our case, we can take 
\[ B= (1, 0, 0, 1, 0, 0, 0, 0)^T. \]

After Dehn filling with a layered solid torus, the Neumann--Zagier matrix of the result can be obtained explicitly from that of the unfilled manifold and the path in the Farey graph~\cite[Prop.~3.11]{HowieMathewsPurcell}.
The portion of $\NZ$ in the top left corner corresponding to edges and tetrahedra outside of the layered solid torus does not change, nor do the entries of the final two cusp rows, corresponding to curves avoiding the layered solid torus.

\subsection{Ptolemy equations}
When a manifold has one cusp, the $n$ edge rows of $\NZ$ have rank $n-1$, so we can remove a row, and the edge rows in the resulting matrix $\NZ^\flat$ are linearly independent \cite{NeumannZagier}.
Denote the vector obtained from $C$ by removing the corresponding row by $C^{\flat}$. This can be done so that one of the first $n-1$ entries of $C^{\flat}$ is nonzero~\cite[Lem.~2.51]{HowieMathewsPurcell}, and we assume our choice has been made so this holds.

The equations defining the A-polynomial involve variables $\gamma_1, \ldots, \gamma_n$ associated to the edge classes $E_1, \ldots, E_n$ of the triangulation.  
Index the edges of each tetrahedron $\Delta_j$ by the ideal vertices at their ends. For $\alpha \beta \in \{01, 02, 03, 12, 13, 23\}$, let $j(\alpha \beta)$ be the index $k$ of the edge $E_k$ to which the edge $\alpha \beta$ of $\Delta_j$ is identified.

\begin{theorem}[\cite{HowieMathewsPurcell}, Theorem~1.1] \label{Thm:HMP}
Let $X$ be a one-cusped manifold with a hyperbolic triangulation $\TT$, with tetrahedra $\Delta_1, \ldots, \Delta_n$, $\NZ^{\flat}$, $C^{\flat}$ and $B=(B_1, B_1', \ldots, B_n, B_n')^T$ as above. 
Denote the entries of the $\mathfrak{m}$ and $\mathfrak{l}$ rows of $\NZ^\flat$ in the $\Delta_j$ columns by $\mu_j, \mu_j'$ and $\lambda_j$, $\lambda_j'$ respectively.

For each tetrahedron $\Delta_j$ of $\TT$, the \emph{Ptolemy equation} of $\Delta_j$ is
\begin{equation}\label{Eqn:Ptolemy}
\left( -1 \right)^{B'_j} \ell^{-\mu_j/2} m^{\lambda_j/2} \gamma_{j(01)} \gamma_{j(23)}
+
\left( - 1 \right)^{B_j} \ell^{-\mu'_j/2} m^{\lambda'_j/2} \gamma_{j(02)} \gamma_{j(13)}
-
\gamma_{j(03)} \gamma_{j(12)}
= 0.
\end{equation}

Setting the $\gamma$ variable corresponding to the row removed from $\NZ$ equal to $1$, and eliminating the other $\gamma$ variables, solving the Ptolemy equations for $m$ and $\ell$, we obtain a factor of the $\PSL(2,\CC)$ A-polynomial; this is the same factor obtained in~\cite{Champanerkar:Thesis}. 
\end{theorem}

For 1-cusped manifolds obtained by Dehn filling the Whitehead sister, the edge classes include $E_{3/1}, E_{4/1}, E_{1/0}$ (recall the labels from \refsec{triangulation_of_Whitehead_sister}), as well as additional edge classes that lie within the layered solid torus. Following \cite{HowieMathewsPurcell}, we label these by their corresponding slope in the Farey graph. 

By switching variables, we may obtain (a factor of) the $\SL(2,\CC)$ A-polynomial.

\begin{corollary}[\cite{HowieMathewsPurcell}, Corollary~1.4]
After setting $M = m^{1/2}$ and $L = \ell^{1/2}$, eliminating the $\gamma$ variables from the polynomial Ptolemy equations as above
yields a polynomial in $M$ and $L$ which contains, as a factor, the factor of the $\SL(2,\C)$ A-polynomial describing hyperbolic structures.
\end{corollary}

For a family of manifolds obtained by Dehn filling a fixed parent manifold, some Ptolemy equations are fixed.

\begin{theorem}[\cite{HowieMathewsPurcell}, Theorem~1.5(i)]
\label{Thm:HMPFilling}
Suppose $X$ has two cusps $\mfc_0, \mfc_1$, and is triangulated such that only two tetrahedra meet $\mfc_1$, and generating curves $\mfm_0, \mfl_0$ on $\mfc_0$ avoid these tetrahedra. Then for any Dehn filling on $\mfc_1$ obtained by attaching a layered solid torus, the Ptolemy equations 
  corresponding to tetrahedra lying outside the layered solid torus are fixed, the same as 
for the unfilled manifold $X$.
\end{theorem}

We now apply these results to Dehn fillings of the Whitehead sister. Recall that we are performing Dehn fillings along the boundary component that corresponds to the unknotted component of the $(-2,3,8)$-pretzel link. 

\begin{lemma}\label{Lem:Tetr01}
Let $X$ be obtained by a nondegenerate, hyperbolic Dehn filling of the Whitehead sister. Then the Ptolemy equations corresponding to tetrahedra 0 and 1 are as follows:
\[ \ell^{-1/2}m^{1/2}\gamma_{1/0}\gamma_{4/1} -\ell^{-1/2}m\gamma_{4/1}\gamma_{3/1} - \gamma_{3/1}^2 =0 \]
\[ -\ell^{-1/2}m\gamma_{3/1}^2 + m^{1/2}\gamma_{1/0}\gamma_{4/1} - \gamma_{3/1}\gamma_{4/1} =0 \]
\end{lemma}
Substituting $\ell=L^2$ and $m=M^2$ in \reflem{Tetr01} yields the $\SL(2,\CC)$ equations.

\begin{proof}[Proof of \reflem{Tetr01}]
Tetrahedra 0 and 1 lie outside the layered solid torus in any Dehn filling, and the triangulation of the Whitehead sister satisfies the requirements of \refthm{HMPFilling}. Thus the Ptolemy equations of tetrahedra 0 and 1 satisfy the conclusions of that theorem, and we may read the equations off of \eqref{Eqn:Ptolemy} using the Neumann--Zagier matrix and $B$-vector computed above for the Whitehead sister. 

For tetrahedron 0, we have $(\mu_0,\mu_0')=(1,1)$ and $(\lambda_0,\lambda_0')=(1,2)$ from the $\NZ$ matrix, so the corresponding Ptolemy equation is
\[ (-1)^{0}\ell^{-1/2}m^{1/2}\gamma_{0(01)}\gamma_{0(23)} + (-1)^{1}\ell^{-1/2}m^{2/2}\gamma_{0(02)}\gamma_{0(13)} - \gamma_{0(03)}\gamma_{0(12)} =0. \]

The following edges are identified to edge classes $E_{3/1}$, $E_{4/1}$ and $E_{1/0}$, respectively:
\[ 0(13),0(12),0(03) \sim E_{3/1}, \qquad 0(02),0(23)\sim E_{4/1}, \quad\text{ and }\quad 0(01)\sim E_{1/0}. \] 
Hence we obtain the Ptolemy equation for tetrahedron 0 as
\[ \ell^{-1/2}m^{1/2}\gamma_{1/0}\gamma_{4/1} -\ell^{-1/2}m\gamma_{4/1}\gamma_{3/1} - \gamma_{3/1}^2 =0. \]

For tetrahedron 1, we similarly obtain the second Ptolemy equation.
%
\end{proof}

\begin{lemma}\label{Lem:Gamma4/1Gamma1/0}
Set $\gamma_{3/1}=1$. Then variables $\gamma_{1/0}$ and $\gamma_{4/1}$ satisfy:
\begin{gather*}
  \gamma_{4/1} = \frac{-\ell+m}{\sqrt{\ell}(-1+m)}, \quad
  \gamma_{1/0} = \frac{-\ell+m^2}{\sqrt{m}(-\ell+m)}
\end{gather*}
\end{lemma}
Again, $SL(2,\C)$ versions can be obtained by substituting $\ell = L^2$ and $m=M^2$.

\begin{proof}
Set $\gamma_{3/1}=1$ and use the equations of the previous lemma. Solving for $\gamma_{0/1}$ and $\gamma_{4/1}$ in terms of $\ell$, $m$, or $L$, $M$ gives the result. 
\end{proof}

\subsection{A-polynomials for Dehn fillings}

We now establish some notation for the slopes in a layered solid torus with reference to the corresponding walk in the Farey triangulation. In the initial step, moving from one triangle to another across an edge, there are four slopes involved. One slope lies on the initial triangle but not the new one; label this $o_0$ (for old). One slope belongs to the new triangle, but not the old; label this $h_0$ (for heading). Two slopes lie on the edge shared by both triangles: label the one to the left $f_0$ and the one to the right $p_0$. For the $k$th step ($k>0$), again label the old slope $o_k$ and the new/heading slope $h_k$. Label the slope around which the $k$th step pivots $p_k$ (for pivot), and the slope that fans out around the pivot $f_k$ (for fan). See \reffig{SlopeLabels}.

\begin{figure}
    \centering
    \includegraphics[width=0.8\textwidth]{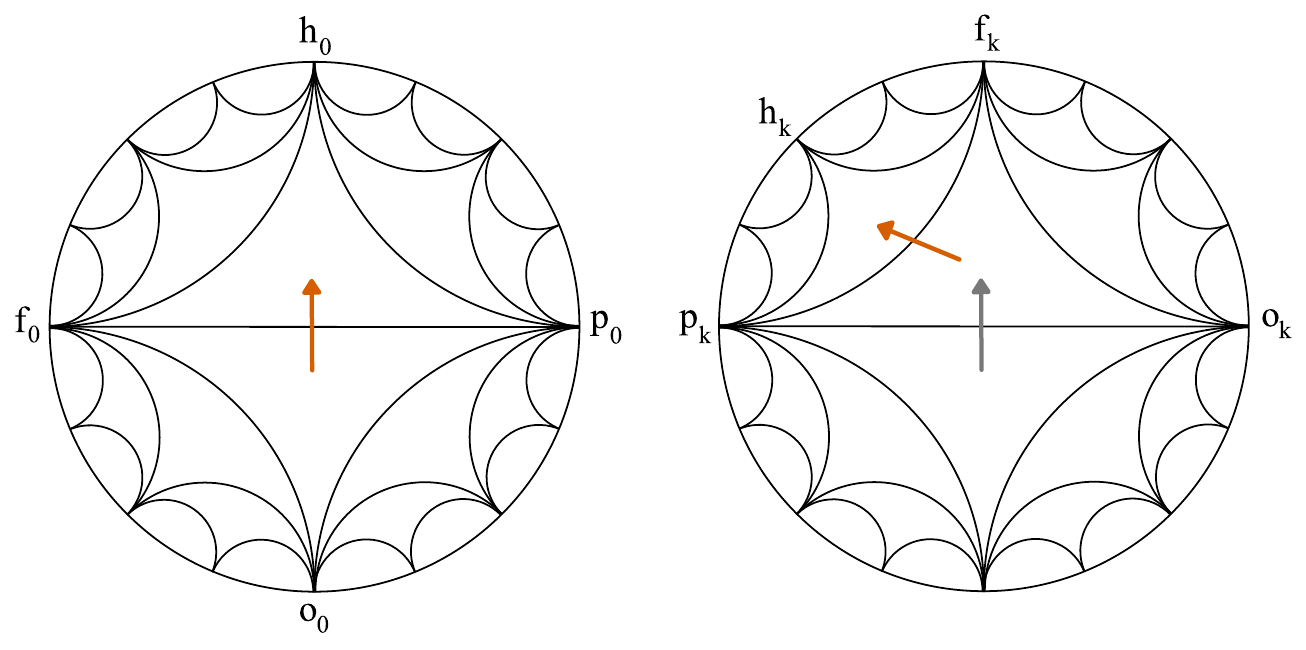}
    \caption{Slope labels for the $k$th step. Left: $k=0$. Right: $k>0$.}
    \label{Fig:SlopeLabels}
\end{figure}

Each edge class in the layered solid torus has a slope $p/q$ and we label the corresponding variable $\gamma_{p/q}$. Theorem~3.17(ii) of \cite{HowieMathewsPurcell} can then be stated as follows, where the length of the walk in the Farey triangulation is denoted $N$.

\begin{theorem}[\cite{HowieMathewsPurcell}, Theorem~3.17(ii)]
\label{Thm:HMPFilling2}
With notation as above, the Ptolemy equations for tetrahedra in a layered solid torus are 
\[ \gamma_{o_k}\gamma_{h_k}+\gamma_{p_k}^2-\gamma_{f_k}^2=0, \text{ for } 0\leq k\leq N-1. \]
When $k=N$ we have the \textit{folding equation} $\gamma_{p_N}=\gamma_{f_N}$.
\end{theorem}

\reflem{Gamma4/1Gamma1/0} and \refthm{HMPFilling2} then immediately yield the following corollary, which gives A-polynomials of Dehn fillings explicitly, as described in the introduction.
\begin{corollary}\label{Cor:DehnFillWhSis}
Let $X$ denote the $p/q$ Dehn filling of the Whitehead sister; $p/q\notin\{2, 3, 7/2, 11/3, 4, 5, 1/0 \}$. Let the corresponding walk in the Farey triangulation, from the triangle $3,4,1/0$ to $p/q$, have length $N$. Then with notation as above, 
(a factor of) the $\PSL(2,\CC)$ A-polynomial of $X$ is given by the following equations.

\noindent  Outside equations:
\[\gamma_{4/1} = \frac{-\ell+m}{\sqrt{\ell}(-1+m)}, \quad
 \gamma_{1/0} = \frac{-\ell+m^2}{\sqrt{m}(-\ell+m)}, \quad \gamma_{3/1}=1
\]
Recursive equations:
\[ \gamma_{h_k} = \frac{\gamma_{f_k}^2-\gamma_{p_k}^2}{\gamma_{o_k}}, \mbox{ for } k=0, \dots N-1, \]
Folding equation: $\gamma_{p_N}=\gamma_{f_N}$. 
\qed
\end{corollary}

Note that the $\gamma$ variables in \refcor{DehnFillWhSis} arise as exponentials of other variables in \cite{HowieMathewsPurcell} (that arise from symplectic linear algebra). As exponentials, they will never be zero, allowing us to write the recursive equations with $\gamma_{o_k}$ in the denominator.

We are particularly interested in $1/n$-Dehn fillings. As seen in \reffig{FareyWhSis}, the first three tetrahedra in the layered solid torus are the same for all $n$, and \refthm{HMPFilling2} immediately yields the following.

\begin{lemma}\label{Lem:First3LSTTet}
Let $K(n)$ be the knot obtained by the $1/n$-Dehn filling of the Whitehead sister.
The Ptolemy equations for the first three tetrahedra in the layered solid torus are as follows. 
\[
\pushQED{\qed}
\gamma_{2/1} = (\gamma_{1/0}^2-\gamma_{3/1}^2)/\gamma_{4/1}, \quad
\gamma_{1/1} = (\gamma_{2/1}^2 - \gamma_{1/0}^2)/\gamma_{3/1}, \quad 
\gamma_{0/1} = (\gamma_{1/1}^2-\gamma_{1/0}^2)/\gamma_{2/1}.  \qedhere
\popQED
\]
\end{lemma}

We can now prove the main result, giving equations for $1/n$ fillings with $n \geq 3$.

\begin{proof}[Proof of \refthm{Main}]
\reflem{Gamma4/1Gamma1/0} gives the  Ptolemy equations for the tetrahedra outside the layered solid torus, and \reflem{First3LSTTet} applies similarly to the first three tetrahedra in the layered solid torus. These tetrahedra correspond to the path in the Farey graph corresponding to $1/n$-Dehn filling, up to the triangle $(1/0,1/1,0/1)$. Moving then to the triangle $(1/2,1/1,0/1)$ yields the equation $\gamma_{1/0}\gamma_{1/2}+\gamma_{1/1}^2-\gamma_{0/1}^2=0$. Moving to the triangle $(0/1, 1/n,1/(n-1))$ via triangles $(0/1, 1/(k-1), 1/(k-2))$ with pivot slope always $0/1$, slope $h_n$ corresponding to $1/k$, old slope $1/(k-2)$, and fan slope $1/(k-1)$, as in \reffig{FareyWhSis}, then gives the recursive inside equations as per \refthm{HMPFilling2}. 
\end{proof}

Now consider $1/n$-Dehn filling for a negative integer $n$. The first three steps in the Farey graph are still the same as the positive case, but at that point the walk in the Farey graph diverges. We obtain the following.

\begin{theorem}\label{Thm:RecursiveLSTNegative}
Let $K(n)$ denote the knot obtained by the $1/n$-Dehn filling of the Whitehead sister. 
For $n\geq 2$, the equations defining the A-polynomial of $K(-n)$ consist of the equations of \reflem{Gamma4/1Gamma1/0}, of \reflem{First3LSTTet}, the equation
\[
\gamma_{1/1}\gamma_{-1/1}+\gamma_{1/0}^2-\gamma_{0/1}^2 =0,
\]
the folding equation $\gamma_{0/1} = \gamma_{-1/(n-1)}$, and for $n\geq 3$, recursive formulae
\[
\gamma_{-1/(k-1)}\gamma_{-1/(k-3)} + \gamma_{0/1}^2 - \gamma_{-1/(k-2)}^2=0, \mbox{ for } 3\leq k \leq n.
\]
As each $\gamma_{p/q}$ can be written in terms of $\ell$ and $m$, or $L$ and $M$, substitution gives (a factor of) the $\PSL(2,\CC)$ or  $\SL(2,\CC)$ A-polynomial.
\end{theorem}

\begin{proof}
Similarly to \refthm{Main}, after initial steps to $(1/0,1/1,0/1)$, the Farey path moves to $(1/0,-1/1, 0/1)$, picking up $\gamma_{1/1}\gamma_{-1/1}+\gamma_{1/0}^2-\gamma_{0/1}^2 =0$. It then moves to $(0/1, -1/n,-1/(n-1))$ by way of triangles $(0/1, -1/(k-1), -1/(k-2))$ with pivot slope
$0/1$, heading slope $-1/k$, old slope $-1/(k-2)$, and fan slope $-1/(k-1)$.
\end{proof}

\subsection{Changing basis}
For our A-polynomial to agree with other computations, we adjust the generators for cusp homology, which were chosen
to have a small number of nonzero entries in the Neumann--Zagier matrix and avoid the hexagon corresponding to the layered solid torus. We now convert to the usual meridian and preferred longitude.

\begin{proposition}\label{Prop:CoB1}
Let $\mathfrak{l}$ and $\mathfrak{m}$ be generators of the cusp homology. A change of basis described by $(\mathfrak{l},\mathfrak{m})\mapsto(\mathfrak{l}^a \mathfrak{m}^b, \mathfrak{l}^c \mathfrak{m}^d)$ corresponds to a change of basis in the variables $\ell,m$ described by $(\ell,m)\mapsto(\ell^d m^{-b}, \ell^{-c} m^a)$. Moreover, after making the substitutions $\ell=L^2$ and $m=M^2$, the change of basis corresponds to $(L,M)\mapsto(L^d M^{-b}, L^{-c} M^a)$.
\end{proposition}
\begin{proof}
Suppose the rows of the $\NZ$ matrix corresponding to $\mathfrak{l}$ and $\mathfrak{m}$ are 
\[ \left[ \lambda_0 \quad \lambda_0' \quad \dots \quad \lambda_{n-1} \quad \lambda_{n-1}' \right]  \text{ and } \left[ \mu_0 \quad \mu_0' \quad \dots \quad \mu_{n-1} \quad \mu_{n-1}' \right]. \]
Then after the change of basis the rows of the $\NZ$ matrix becomes 
\begin{align*}
  &\left[ a\cdot\lambda_0+b\cdot\mu_0 \quad a\cdot\lambda_0'+b\cdot\mu_0' \quad \dots \quad a\cdot\lambda_{n-1}+b\cdot\mu_{n-1} \quad a\cdot\lambda_{n-1}'+b\cdot\mu_{n-1}' \right] \text{ and } \\
    &\left[ c\cdot\lambda_0+d\cdot\mu_0 \quad c\cdot\lambda_0'+d\cdot\mu_0' \quad \dots \quad c\cdot\lambda_{n-1}+d\cdot\mu_{n-1} \quad c\cdot\lambda_{n-1}'+d\cdot\mu_{n-1}' \right]. 
\end{align*}
The $C$ vector also changes accordingly, and the same $B$ vector satisfies the new equation $\NZ\cdot B=C$. As such, the coefficient of $\gamma_{i(01)}\gamma_{i(23)}$  becomes
\begin{align*}
  (-1)^{B_i'}\ell^{-(c\lambda_i+d\mu_i)/2}m^{(a\lambda_i+b\mu_i)/2}\gamma 
  &= (-1)^{B_i'}(\ell^{d}m^{-b})^{-\mu_i/2} (\ell^{-c}m^a)^{\lambda_i/2}.
\end{align*}
Similar reasoning shows that the coefficient of $\gamma_{i(02)}\gamma_{i(13)}$ becomes
\[ (-1)^{B_i}(\ell^{d}m^{-b})^{-\mu_i'/2}(\ell^{-c}m^a)^{\lambda_i'/2}. \]
Thus, the change of basis ${(\mathfrak{l},\mathfrak{m})\mapsto(\mathfrak{l}^a \mathfrak{m}^b, \mathfrak{l}^c \mathfrak{m}^d)}$ 
corresponds to ${(\ell,m)\mapsto(\ell^d m^{-b}, \ell^{-c} m^a)}$ in the A-polynomial variables. A similar argument holds for $L$, $M$.
\end{proof}

\begin{proposition}\label{Prop:CoB2}
Let $K$ be a link in $S^3$ with components $K_1$, $K_2$, where $K_2$ is unknotted. 
Let $\mathfrak{l}$ and $\mathfrak{m}$ be generators of the cusp homology corresponding to $K_1$ and let $\mathfrak{l}'=\mathfrak{l}^a\mathfrak{m}^b$ and $\mathfrak{m}'=\mathfrak{l}^c\mathfrak{m}^d$ be the actual meridian and preferred longitude. Let $x$ be the linking number of $K_1$ and $K_2$. The change of basis required for a $1/n$-Dehn filling is
  \[ (\mathfrak{l},\mathfrak{m})\mapsto (\mathfrak{l}^{a}\mathfrak{m}^{b+nx^2}, \mathfrak{l}^{c}\mathfrak{m}^{d}). \]
\end{proposition}
\begin{proof}
See Rolfsen's textbook~\cite{Rolfsen}, Section 9H; in particular page 267.
\end{proof}

\begin{corollary}\label{Cor:WhiteheadSisCOB}
Let $K(n)$ denote the $1/n$-Dehn filling of the Whitehead sister. 
Then for $K(n)$, the required change of basis from the basis of \reffig{WhSisCuspTri} to the standard meridian and longitude in \reffig{preflong} is given by:
\[ (\mathfrak{l},\mathfrak{m})\mapsto (\mathfrak{l}\mathfrak{m}^{-8+25n}, \mathfrak{m}). \]
Consequently, the change of basis in the A-polynomial variables is:
\[ (\ell, m) \mapsto (\ell m ^{8-25n}, m) \quad \mbox{and} \quad
(L,M)\mapsto (LM^{8-25n},M). \]
\end{corollary}

\begin{proof}
The linking number of the two components is 5, so applying Proposition~\ref{Prop:CoB2}, the required change of basis is as claimed. 
By \refprop{CoB1}, the required change of basis in the A-polynomial variables is also as claimed.
\end{proof}

\section{Appendix: A-polynomial calculations}

In this section, we include calculations of some of the simplest A-polynomials arising from our Dehn fillings.

\subsection{The knot $K3_1$}
Recall that the knot $K3_1$ is obtained by $1/1$-Dehn filling the Whitehead sister. 
Set $\gamma_{3/1}=1$ and use the equations of \reflem{Gamma4/1Gamma1/0} to obtain equations for $\gamma_{4/1}$, $\gamma_{1/0}$ in terms of $(\ell, m)$ or $(L,M)$. 

The $1/1$-Dehn filling is obtained by attaching only one tetrahedron in the layered solid torus, and then folding; see \reffig{FareyWhSis}. This gives two equations: one Ptolemy equation $\gamma_{4/1}\gamma_{2/1}+\gamma_{3/1}^2-\gamma_{1/0}^2=0$
and the folding equation $\gamma_{2/1}=\gamma_{1/0}$.

In terms of $L$ and $M$, plugging the folding equation into the Ptolemy equation, as well as equations of $\gamma_{4/1}$, $\gamma_{1/0}$, $\gamma_{3/1}=1$, gives
\[ -\frac{\left(L^2-M^4\right)^2}{M^2 \left(L^2-M^2\right)^2}+\frac{L^2-M^4}{M \left(L-L M^2\right)}+1 \]

After applying the change of basis of \refcor{WhiteheadSisCOB}, the largest factor is
\begin{align*}
   L^6-L^5 M^{20}+2 L^5 M^{18}-L^5 M^{16}-L^4 M^{38}-2 L^4 M^{36}+
   2 L^2 M^{74}+L^2 M^{72}  \\ +L M^{94}-2 L M^{92}+L M^{90}-M^{110}.
\end{align*}
This is identical to Culler's A-polynomial for $K3_1$ \cite{Culler}. 

We may instead use the expressions for $\gamma_{4/1}$ and $\gamma_{1/0}$ in terms of $\ell$ and $m$, to obtain the $\PSL(2,\CC)$ A-polynomial:
\begin{align*}
  \ell^3 - \ell^{5/2} m^{10} + 2 \ell^{5/2} m^9 - \ell^{5/2} m^8 - \ell^2 m^{19} -
  2 \ell^2 m^{18} + 2 \ell m^{37} + \ell m^{36}  \\
  + \sqrt{\ell} m^{47} - 2 \sqrt{\ell} m^{46} + \sqrt{\ell} m^{45} - m^{55}
\end{align*}

\subsection{The knot $K5_4$}
The knot $K5_4$ is obtained by $1/2$-Dehn filling the Whitehead sister. This requires attaching two ideal tetrahedra in the layered solid torus, and then folding; see \reffig{FareyWhSis}.
To compute the A-polynomial for $K5_4$, we use all the equations we used for $K3_1$ except for the folding equation, along with two new Ptolemy equations corresponding to steps 1 and 2. The new equations are
\[
  \gamma_{3/1}\gamma_{1/1}+\gamma_{1/0}^2-\gamma_{2/1}^2 =0 \quad \text{and} \quad 
  \gamma_{2/1}\gamma_{0/1}+\gamma_{1/0}^2-\gamma_{1/1}^2 =0. 
\]
The folding equation for $K5_4$ is $\gamma_{0/1}=\gamma_{1/1}$. 

Using the expressions for $\gamma_{4/1}$ and $\gamma_{1/0}$ in terms of $L$ and $M$ (with $\gamma_{3/1}$ set to 1), we find a precursor to the A-polynomial of the $K5_4$ knot of the form
\begin{align*}
  -\frac{1}{M^8 \left(L^2-M^2\right)^{12}} -M^6 \left(L^2-M^4\right)^2 \left(L^2-M^2\right)^{10} \\
  +\left(\left(L^2-M^2\right)^4 \left(M^5-L^2 M \right)^2-\left(M^2-1\right)^4 \left(L^5-L M^6\right)^2\right)^2\\
  +L \left(M-M^3 \right)^2 \left( L^4-M^6 \right) \left(L^2-M^2\right)^3 \\
  \left( \left(L^2-M^2\right)^4 \left(M^5-L^2 M\right)^2
  -\left(M^2-1\right)^4 \left(L^5-L M^6\right)^2\right) 
\end{align*}

After applying the change of basis of \refcor{WhiteheadSisCOB} and clearing negative exponents, the largest factor is a polynomial identical to Culler's A-polynomial for $K5_4$ \cite{Culler}. We omit the polynomial here; it has 106 terms, maximum degree 820 in $M$, 19 in $L$. 

\bibliography{biblio}
\bibliographystyle{amsplain}

\end{document}